\documentclass[11pt]{article}
\usepackage{fullpage}
\usepackage{latexsym}
\usepackage{amssymb}
\usepackage{amsthm}
\usepackage{amscd}
\usepackage{amsmath}
\usepackage{mathrsfs}
\usepackage{graphicx}
\usepackage{shuffle}
\usepackage[colorlinks=true, pdfstartview=FitV, linkcolor=blue, citecolor=blue, urlcolor=blue]{hyperref}
\usepackage{mathtools}
\usepackage[bbgreekl]{mathbbol}
\usepackage{mathdots}
\usepackage{ytableau}
\usepackage[enableskew,vcentermath]{youngtab}
\usepackage{tikz-cd}
\usepackage[title]{appendix}
\usepackage{mathtools}
\usepackage{amssymb}
\usepackage{cleveref}
\usepackage[backend=bibtex,giveninits,sorting = nyt,sortcites]{biblatex}

\usepackage[all]{xy}
\input xy \xyoption{frame}
\xyoption{dvips}

\usepackage[colorinlistoftodos]{todonotes}
\usetikzlibrary{chains,scopes}
\usetikzlibrary{shapes.geometric,positioning}

\usepackage{enumitem}

\DeclareSymbolFontAlphabet{\mathbb}{AMSb}
\DeclareSymbolFontAlphabet{\mathbbl}{bbold}

\def\YR{\ensuremath\mathrm{YR}}

\newcommand{\rfc}{\mathrm{BRF}}

\newcommand{\rf}{\mathrm{RF}}

\newcommand{\SSYT}{\ensuremath\mathrm{SSYT}}
\newcommand{\SSKT}{\ensuremath\mathrm{SSKT}}

\definecolor{darkred}{rgb}{0.7,0,0} 
\newcommand{\defn}[1]{{\color{darkred}\emph{#1}}} 

\renewcommand{\mid}{:}

\numberwithin{equation}{section}
\allowdisplaybreaks[1]
\UseCrayolaColors

\theoremstyle{definition}
\newtheorem* {theorem*}{Theorem}
\newtheorem* {corollary*}{Corollary}
\newtheorem* {conjecture*}{Conjecture}
\newtheorem{theorem}{Theorem}[section]

\theoremstyle{definition}

\newtheorem* {example*}{Example}

\newtheorem{lemma}[theorem]{Lemma}
\theoremstyle{definition}
\newtheorem{definition}[theorem]{Definition}
\theoremstyle{definition}

\newtheorem{proposition}[theorem]{Proposition}
\newtheorem{corollary}[theorem]{Corollary}

\newtheorem{remark}[theorem]{Remark}
\newtheorem*{remark*}{Remark}
\theoremstyle{definition}
\newtheorem {example}[theorem]{Example}
\theoremstyle{definition}

\theoremstyle{definition}

\theoremstyle{definition}

\xyoption{dvips}

\newcommand{\rowreading}{\mathsf{row}}
\newcommand{\shift}{\mathsf{shift}}

\newcommand{\weakdestab}{\mathsf{WeakDesTab}}

\def\({\left(}
\def\){\right)}

\newcommand{\cC}{\mathcal{C}}

\def\ZZ{\mathbb{Z}}

\def\ch{\operatorname{ch}}

\def\sh{\mathrm{sh}}

\def\fk{\mathfrak}

\def\barr{\begin{array}}
\def\earr{\end{array}}
\def\ba{\begin{aligned}}
\def\ea{\end{aligned}}
\def\be{\begin{equation}}
\def\ee{\end{equation}}

\def\qquand{\qquad\text{and}\qquad}
\def\quand{\quad\text{and}\quad}

\newcommand{\gl}{\mathfrak{gl}}

\def\id{\mathrm{id}}

\def\ben{\begin{enumerate}}
\def\een{\end{enumerate}}

\def\bei{\begin{itemize}}
\def\eei{\end{itemize}}

\def\des{\mathrm{des}}

\def\e{\textbf{e}}

\newcommand{\stdk}{\mathsf{std_{key}}}

\newcommand{\cB}{\mathcal{B}}

\def\arcstart{\ \xy<0cm,-.15cm>\xymatrix@R=.1cm@C=.3cm }
\newcommand{\arcstartc}[1]{\ \xy<0cm,-.15cm>\xymatrix@R=.1cm@C=#1cm}

\def\sW{\mathscr{W}}

\def\ba{\mathbf{a}}

\newcommand{\row}{\operatorname{row}}

\newcommand{\weight}{\mathsf{wt}}

\def\row{\textsf{row}}

\def\path{\mathsf{path}_{\textsf{OEG}}}

\def\whSym
{\mathfrak{m}\textsf{Sym}}

\def\whQSym
{\mathfrak{m}\textsf{QSym}}

\def\fkD{\fk D}

\def\rf{\mathrm{RF}}

\numberwithin{equation}{section}
\allowdisplaybreaks[1]
\UseCrayolaColors

\def\lift{\mathsf{lift}}

\usepackage{listings}
\lstdefinelanguage{Sage}[]{Python}
{morekeywords={False,sage,True},sensitive=true}
\definecolor{dblackcolor}{rgb}{0.0,0.0,0.0}
\definecolor{dbluecolor}{rgb}{0.01,0.02,0.7}
\definecolor{dgreencolor}{rgb}{0.2,0.4,0.0}
\definecolor{dgraycolor}{rgb}{0.30,0.3,0.30}

\ytableausetup{aligntableaux=bottom}

\begin{document}
\title{Demazure crystals for flagged key polynomials}
\author{Jiayi Wen \\ Department of Mathematics \\ Hong Kong University of Science and Technology \\ {\tt jiayi.wen@connect.ust.hk
}}

\date{}

\maketitle

\begin{abstract} 
    One definition of key polynomials is as the weight generating functions of key tableaux. 
    Assaf and Schilling introduced a crystal structure on key tableaux and related it to Morse--Schilling crystals on reduced factorizations for permutations via weak Edelman--Greene insertion. In this paper, we consider generalizations of both crystals depending on a flag. We extend weak EG insertion to a bijection between our flagged objects and show that the recording tableau gives a crystal isomorphism. As an application, we show that flagged key tableaux have a natural Demazure crystal structure, whose characters recover Reiner and Shimozono's flagged key polynomials.
\end{abstract}

\section{Introduction}

\defn{Schur polynomials} $s_\lambda$, indexed by partitions, are the weight generating functions for
\defn{semistandard Young tableaux} of shape $\lambda$.
As $\lambda$ varies over all partitions with at most $n$ parts, $\{s_\lambda\}$ forms a $\ZZ$-basis for the ring of
symmetric polynomials in $n$ variables. In representation theory, Schur polynomials are characters
of irreducible polynomial representations of the general linear groups.
\defn{Key polynomials} $\kappa_\alpha$, indexed by weak compositions, are the characters of Demazure modules, which are representations of Borel subgroups of the general linear groups \cite{Dem}. 
Key polynomials are nonsymmetric generalizations of Schur polynomials, and
$\{\kappa_\alpha\}$
forms a basis for $\ZZ[x_1,x_2,\dots]$ as $\alpha$ ranges over all weak compositions.
Lascoux and Sch\"utzenberger proved that key polynomials are
weight generating functions of various
combinatorial objects \cite{LS2}.

Key polynomials are also related to other important polynomials in combinatorics.
For example, \defn{Schubert polynomials} are
representatives of the Schubert classes in the cohomology ring of the complete flag variety \cite{LS}. Schubert polynomials also generalize Schur polynomials in the sense that Schur polynomials are representatives of the Schubert classes indexed by \defn{Grassmannian} permutations.
Schubert polynomials are also weight generating functions for \defn{bounded reduced factorizations},
which are certain generalizations of reduced words for a permutation \cite{BJS, FS}; see \S\ref{section 3} for more details.
Lascoux and Sch\"utzenberger proved that Schubert polynomials expand as non-negative integer linear combinations of key polynomials; see \cite[Thm.~3]{RS}.

Assaf generalized semistandard Young tableaux to provide another combinatorial formula of key polynomials \cite{Assaf18}.
Namely, $\kappa_\alpha$ is the weight generating function for \defn{(semistandard) key tableaux} of shape $\alpha$.
In the later paper \cite{Assaf19}, Assaf introduced a \defn{weak Edelman--Greene (EG) insertion algorithm}
to give a new bijective proof of the key positivity of Schubert polynomials.
Assaf's algorithm is based on the \defn{Edelman--Greene insertion}, which Edelman and Greene introduced in \cite{EG} to prove that \defn{Stanley symmetric functions} \cite{S} have non-negative integer coefficients in the Schur basis.

Let $\rf_n(w)$ be the set of all reduced factorizations for a permutation $w$ with exactly $n$ components.
In \cite{MS}, Morse and Schilling defined a $\gl_n$-crystal 
structure on $\rf_n(w)$,
which they related to a crystal structure on semistandard Young tableaux through EG insertion.
In \cite{AS}, Assaf and Schilling defined a crystal on key tableaux
and proved that $\SSKT(\alpha)$, the set of all key tableaux with shape $\alpha$,
is the Demazure truncation of the highest weight crystal of the highest weight $\lambda$,
where $\lambda$ is the partition rearrangement of $\alpha$. 
Then they proved that the recording tableau for weak EG insertion
gives a crystal isomorphism between $\rfc(w)$ and $\bigsqcup_\alpha\SSKT(\alpha)$,
where $\rfc(w) \subseteq \rf_n(w)$ is the subset of reduced factorizations that are bounded in the sense of \ref{rfc-def}.

A \defn{flag} is a weakly increasing function $\varphi:\ZZ_{>0}\to \ZZ_{>0}$ such that $\varphi(n)\geq n$ for all $n\in\ZZ_{>0}$.
If $\varphi(i) =i$ for all $i$, then we say that $\varphi$ is the \defn{standard flag}.
In both $\rfc(w)$ and $\SSKT(\alpha)$, the entries in the $i$-th part of the elements are bounded by $i$ which is the $i$-th value of the standard flag.
A natural way of generalizing these sets is to replace the standard flag by an arbitrary flag.
In a \defn{$\varphi$-flagged reduced factorization},
a given letter $i$ may appear in the first $\varphi(i)$ parts instead of the first $i$ parts.
The $i$-th row of a \defn{$\varphi$-flagged key tableau} is bounded above by $\varphi(i)$ instead of $i$.
We denote these $\varphi$-flagged generalizations as $\rfc(w,\varphi)$ and $\SSKT(\alpha,\varphi)$, respectively. 

Our main results concern crystal structures on these sets. 
Let $\varphi$ be any flag.
The set $\rfc(w,\varphi)$ inherits a crystal structure from \cite{MS}, and we show how to 
extend the existing crystal structure on $\SSKT(\alpha)$ to the larger set $\SSKT(\alpha,\varphi)$. 
Assaf's weak EG insertion algorithm can be evaluated on flagged factorizations, and we are able to prove the following generalization of \cite[Thm.~5.10]{AS}:

\begin{proposition}[{See Propositions~\ref{flagged bijection} and \ref{cor 4.6}}]
\label{theorem1}
    The weak EG insertion algorithm induces a crystal isomorphism from $\rfc(w,\varphi)$ to $\bigsqcup_\alpha \SSKT(\alpha,\varphi)$.

\end{proposition}

One of our other main results gives an explicit relation between the crystals $\rfc(w,\varphi)$ and  $\rfc(w)$.
Let $\cB$ be any crystal with raising operators $e_i:\cB\to \cB\sqcup\{0\}$.
The $i$-th \defn{Demazure operator} $\fkD^{\cB}_{i}$ acts on subsets $X\subseteq \cB$ by  
$\fkD^{\cB}_i X:=\{b\in \cB\mid e_i^k(b)\in X\text{ for some }k\geq 0\}$.
When $\cB=\rf_n(w)$,
applying arbitrary sequences of these operators to highest weight elements  
generates a family of subcrystals called \defn{Demazure crystals}; see Definition~\ref{dem-def}.

\begin{theorem}[{See Theorem~\ref{thm 44}}]\label{theorem2}
  Let $n$ be sufficiently large such that $\rfc(w,\varphi)\subseteq \rf_n(w)$.
   Then $\rfc(w,\varphi)=\fkD_{i_1}^{\rf_n(w)}\cdots \fkD_{i_k}^{\rf_n(w)}\rfc(w)$ for certain 
    $i_1,\dots,i_k\in \ZZ_{>0}$ depending only on $\varphi$.
\end{theorem}

    \begin{corollary}\label{cor 1.3}
  Each $\rfc(w,\varphi)$ is a disjoint union of Demazure crystals. 
    \end{corollary}

    It is shown in \cite{AS} that $\SSKT(\alpha)$ is a Demazure crystal. We can upgrade this to any flag:

\begin{theorem}[{See Theorem~\ref{demazure crystal sskt}}]\label{thm1.3}
Each $\SSKT(\alpha,\varphi)$ is a Demazure crystal.
\end{theorem}

Taking the character of $\SSKT(\alpha,\varphi)$ gives a flagged generalization of key polynomials. One reason that our flagged constructions are natural to consider is that 
such \defn{flagged key polynomials} have already appeared in the literature in a different form. 
Reiner and Shimozono considered flagged key polynomials in the context of a flagged Littlewood-Richardson rule \cite[Thm.~20]{RS}.
Their original definition does not involve key tableaux; however, we can prove the following:

\begin{theorem}
\label{thm 1.4}
    Reiner and Shimozono's flagged key polynomial $\kappa_{(\alpha,\varphi)}$ from \cite{RS} is equal to $\sum_{T\in \SSKT(\alpha,\varphi)}x^{\weight(T)}$.  
\end{theorem}

Reiner and Shimozono observed that  $\kappa_{(\alpha,\varphi)}=\kappa_\beta$ for 
some $\beta$. Our results imply a crystal analog:

\begin{corollary}[{See Corollary~\ref{thm 4.15}}]
    The Demazure crystal $\SSKT(\alpha,\varphi)\cong\SSKT(\beta)$ for some $\beta$.
\end{corollary}
    
Other recent work \cite{Ku} has also considered crystal structures on flagged objects such as flagged reversed plane partitions; however, these objects are different from what we consider in this paper. 

The structure of the paper is as follows.
In \S\ref{section 2}, we review the definitions of (flagged) key polynomials and key tableaux.
In \S\ref{section 3}, we review Assaf's weak EG insertion algorithm and
extend it to arbitrary flags.
In \S\ref{application}, we review the Morse--Schilling crystal on $\rf_n(w)$ and extend the Assaf--Schilling crystal operators to $\SSKT(\alpha,\varphi)$. Our main results are proved in \S\ref{application}.

\section{Preliminaries}\label{section 2}

In this section, we review some preliminaries on key polynomials and key tableaux from \cite{AS, RS, LS2}. 

\subsection{Key polynomials}\label{2.1}
Let $S_\infty$ be the group of permutations of $\ZZ_{>0}=\{1,2,\cdots\}$
that fix all but finitely many points.
Write $s_i \in S_\infty$ for the adjacent transposition that exchanges $i$ and $i+1$.
A \defn{word} is a finite sequence of positive integers.
A minimal-length word $i_1\cdots i_n$ with $w=s_{i_1}\cdots s_{i_n}$ is a \defn{reduced word} for $w \in S_\infty$.
Let $R(w)$ be the set of all reduced words for $w$.
The \defn{length} of a permutation $w \in S_\infty$ is the common length of any of its reduced words.
In this article, whenever we say ``reduced word'' we mean an element of $R(w)$ for some $w \in S_\infty$.

The group $S_\infty$ acts on $\ZZ[x_1,x_2,\dots]$
by permuting indeterminates.
One has \defn{divided difference operators} $\partial_i$ and $ \pi_i$ on
$\ZZ[x_1,x_2,\dots]$ by the formulas $\partial_i (f):=\frac{f-s_i\cdot f}{x_i-x_{i+1}}$ and $\pi_i (f)
    :=\partial_i (x_if)$.
When $f$ is symmetric in $x_i$ and $x_{i+1}$ (so that $s_i\cdot f=f$),
we have $\partial_i(f)=0$ and $\pi_i(f)=f$.
Both divided difference operators satisfy the braid relations for $S_\infty$ along with $\partial_i^2=0, \ \pi_i^2=\pi_i$ for all $i$.
This means that we can define $\partial_w:=\partial_{i_1}\cdots\partial_{i_n}$
and $\pi_w:=\pi_{i_1}\cdots\pi_{i_n}$ for any $i_1\cdots i_n\in R(w)$. 

For $m\in \ZZ_{\geq0}$,
a \defn{weak composition} $\alpha$ of $m$ is an infinite tuple of non-negative integers
$(\alpha_1,\alpha_2,\dots)$ with finitely many $\alpha_i>0$ and $\sum_{i=1}^\infty\alpha_i=m$.
We identify finite tuples $(\alpha_1,\alpha_2,\dots,\alpha_n)$ with the zero-padded infinite tuples $(\alpha_1,\alpha_2,\dots,\alpha_n,0,0,\dots)$.
The \defn{length} $\ell(\alpha)$ is the minimal $n\in \ZZ_{\geq 0}$ such that $\alpha_i=0$ for all $i>n$.
The group $S_\infty$ acts on weak compositions on the right by permuting components.
A \defn{partition} $\lambda=(\lambda_1,\lambda_2,\dots)$ is a weak composition
such that $\lambda_1\geq \lambda_2\geq \cdots$.

Given a weak composition $\alpha$,
let $w$ be the minimal-length permutation such that $\alpha\cdot w=\lambda$,
where $\lambda$ is a partition.
The \defn{key polynomial} associated to $\alpha$ is $\kappa_\alpha: =\pi_{w}x^{\lambda}$,
where $x^\lambda=x_1^{\lambda_1}x_2^{\lambda_2}\cdots$.
For example, $\kappa_{1201}=\pi_1\pi_3(x^{2110})=x^{2110}+x^{1210}+x^{2101}+x^{1201}$.
We often use the following recursive property of key polynomials: if $\alpha_i>\alpha_{i+1}$ then $\pi_i(\kappa_\alpha) =
    \kappa_{\alpha\cdot s_i}$ and otherwise $\pi_i(\kappa_\alpha) =
    \kappa_{\alpha}$.

\subsection{Key tableaux}
A \defn{diagram} is a finite set of left-aligned boxes in the right half plane $\ZZ \times \ZZ_{>0}$.
Most of the diagrams appearing in this paper will be in $\ZZ_{>0}\times\ZZ_{>0}$.
If not, we draw a horizontal line between the rows indexed by $0$ and $1$. 
Throughout this paper, we use French notation when we draw diagrams,
where the columns are indexed from left to right and the rows are indexed from bottom to top.
The \defn{shape} of a diagram is the weak composition recording its number of boxes in each row whose index is positive. A \defn{tableau} is a filling of a diagram with positive integers.
For example, the tableau on the left has a row with a non-positive index, and the tableau on the right is in $\ZZ_{>0}\times \ZZ_{>0}$:
\[
    \begin{ytableau}
        2 & 7\\
        3 & 6 & \none & \none\\
        \hline
        2 & 4 & 5
    \end{ytableau}
    \qquad
    \begin{ytableau}
        6\\
        \none[\cdot] \\
        3 & 4\\
        2\\
        \none
    \end{ytableau}.
\]
The \defn{weight} of a word $u$
is the weak composition $\weight(u):=(\beta_1,\beta_2,\dots)$ of $\ell(u)$,
where $\beta_i$ is the multiplicity of $i$ in $u$.
The \defn{row reading word} of a tableau $T$ is the word $\rowreading (T) :=  \cdots r^{(2)}r^{(1)}$,
where $r^{(i)}$ is the $i$-th row of $T$ read from left to right.
The \defn{weight} of a tableau $T$ is $\weight(T):= \weight(\rowreading(T))$.
For example, the left tableau has row reading word $2736245$ and weight $(0,2,1,1,1,1,1)$.

A \defn{semistandard Young tableau} is a tableau of a partition shape whose rows are weakly increasing and whose columns are strictly increasing.
A \defn{standard Young tableau} is a semistandard Young tableau whose entries are the numbers $1,2,\dots,n$ with no repetitions for some integer $n\geq 0$.
The following definition of a \defn{key tableau} is identical to \cite[Def.~3.1]{AS}.
Here, given a tableau $T$ and a box $(i,j)$ which it contains, we write $T_{ij}$ for the entry in box $(i,j)$.

\begin{definition}
    \label{key tableau def}
    A \defn{key tableau} $T$ is a tableau such that
    \begin{enumerate}[label=(\alph*)]
        \item each row is weakly decreasing, and each column has distinct entries; and
        \item if $i< k$ and  $T_{ij}>T_{kj}$,
              then $T$ contains $(i,j+1)$ and $T_{i,j+1} > T_{kj}$.
    \end{enumerate}
\end{definition}
A \defn{standard key tableau} is a key tableau filled by $1,2,\dots,n$ with no repetitions for some $n\in \ZZ_{\geq 0}$.
For example, the   tableaux
\[
    \begin{ytableau}
    \ytableausetup{aligntableaux=center}
        5 & 4 & 3\\
        \none[\cdot]\\
        7 & 6 & 5 & 5\\
        1 & 2
    \end{ytableau}
    \qquand
    \begin{ytableau}
        5 & 4 \\
        3 \\
        2 & 1
    \end{ytableau}
\]
are key tableaux of shapes $(2,4,0,3)$ and $(2,1,2)$.

Fix a flag $\varphi$. We say that a key tableau $T$ is \defn{$\varphi$-flagged} if $T_{ij} \leq \varphi(i)$ whenever $(i,j)$ is a box in $T$.
We denote the set of all $\varphi$-flagged key tableaux of shape $\alpha$ as $\SSKT(\alpha,\varphi)$.
If $\phi$ is another flag such that $\varphi(i)=\phi(i)$ for all $1\leq i\leq \ell(\alpha)$, then $\SSKT(\alpha,\varphi)=\SSKT(\alpha,\phi)$.
When $\varphi$ is the standard flag with $\varphi(i)=i$ for all $i$, we omit $\varphi$ in our notation and write $\SSKT(\alpha,\varphi)$ as $\SSKT(\alpha)$.
The elements of $\SSKT(\alpha)$ are what Assaf and Schilling called \defn{semistandard key tableaux} in \cite[Def.~3.2]{AS}. 

\begin{theorem}[{\cite[Prop.~2.6]{Assaf18}}]\label{key-comb-expansion}
    If $\alpha$ is a weak composition then $\kappa_\alpha =\sum_{T\in \SSKT(\alpha)}x^{\weight(T)}$.
\end{theorem}

\begin{example}
If $\alpha=(1,2,0,1)$ then $\kappa_\alpha=x^{2110}+x^{1210}+x^{2101}+x^{1201}$ while the set $\SSKT(\alpha)$ consists of four key tableaux:
\[
    \begin{ytableau}
        3\\
        \none[\cdot]\\
        2 & 1\\
        1
    \end{ytableau}
    \quad\begin{ytableau}
        3\\
        \none[\cdot]\\
        2 & 2\\
        1
    \end{ytableau}
    \quad\begin{ytableau}
        4\\
        \none[\cdot]\\
        2 & 1\\
        1
    \end{ytableau}
    \quad\begin{ytableau}
        4\\
        \none[\cdot]\\
        2 & 2\\
        1
    \end{ytableau}.
\]
\end{example}
Key polynomials generalize \defn{Schur polynomials} in the following way.
Let $\lambda$ be a partition such that $\ell(\lambda)\leq n$ and let $\SSYT_n(\lambda)$ be the set of semistandard Young tableaux of shape $\lambda$ filled by $1,2,\dots, n$. The \defn{Schur polynomial} indexed by $\lambda$ with $n$ variables is $s_\lambda(x_1,\dots,x_n)=\sum_{T\in \SSYT_n(\lambda)}x^{\weight(T)}$. 
It is well-known that $s_\lambda(x_1,\dots,x_n)$ is $S_n$-invariant \cite[Thm.~7.10.2]{EC2}.

Suppose $\alpha$ is the weak composition $(\lambda_n,\lambda_{n-1},\dots,\lambda_1)$. For this choice of $\alpha$, any $T\in \SSKT(\alpha)$ is weakly decreasing along each row and strictly increasing along each column. 
We have a weight-reversing bijection from $\SSKT(\alpha)$ to $\SSYT_n(\lambda)$ sending a key tableau $T$ to the unique semistandard Young tableau $S$ with $S_{i,j}  = n+1-T_{n+1-i,j}$. Hence
\[
    \kappa_\alpha =     s_\lambda(x_n,x_{n-1},\dots,x_1) = s_\lambda(x_1,\dots,x_n)
\]
and so every Schur polynomial is a key polynomial. 

Theorem~\ref{key-comb-expansion} suggests a natural flagged generalization of key polynomials.
Our main results imply that these flagged generalizations coincide with the flagged key polynomials defined in \cite{RS}.

\begin{definition}
    The \defn{$\varphi$-flagged key polynomial}
    of a weak composition $\alpha$ is $\kappa_{(\alpha,\varphi)}:=\sum_{T\in \SSKT(\alpha,\varphi)}x^{\weight(T)}$.
\end{definition}

\begin{example}\label{flag key example}
Let $\alpha = (1,2,0,1)$ as above. Let $\varphi$ be the non-standard flag such that $\varphi(1)=2$, $\varphi(2)=3$ and $\varphi(3)=\varphi(4)=4$. There are $7$ key tableaux in $\SSKT(\alpha,\varphi)$ that are not in  $\SSKT(\alpha)$:
\[
    \begin{ytableau}
        2\\
        \none[\cdot]\\
        3 & 3\\
        1
    \end{ytableau}
    \quad\begin{ytableau}
        4\\
        \none[\cdot]\\
        3 & 3\\
        1
    \end{ytableau}
    \quad\begin{ytableau}
        4\\
        \none[\cdot]\\
        3 & 2\\
        1
    \end{ytableau}
    \quad\begin{ytableau}
        4\\
        \none[\cdot]\\
        3 & 1\\
        1
    \end{ytableau}
    \quad\begin{ytableau}
        4\\
        \none[\cdot]\\
        3 & 3\\
        2
    \end{ytableau}
    \quad\begin{ytableau}
        4\\
        \none[\cdot]\\
        3 & 2\\
        2
    \end{ytableau}
    \quad\begin{ytableau}
        4\\
        \none[\cdot]\\
        3 & 1\\
        2
    \end{ytableau}.
\]
In this case $\kappa_{(\alpha,\varphi)}=x^{2110}+x^{1210}+x^{2101}+x^{1201}+x^{1120}+x^{1021}+2x^{1111}+x^{2011}+x^{0121}+x^{0211}$.
\end{example}

\section{Insertion algorithms}\label{section 3}

In this section,
we review Assaf's definition of \defn{weak EG insertion} and the \defn{lift operation} \cite{Assaf19}.
The main goal is to prove Proposition~\ref{flagged bijection}, which describes a flagged version of the weak EG correspondence.
Most of the material in this section is either a review of \cite{Assaf19,AS} or a mild generalization of results therein.
However, Corollary~\ref{corollary 331} and Lemma~\ref{flag lemma} were not presented in \cite{AS}.

\subsection{Partial orders and weak EG insertion tableaux}\label{3.2}

Before we describe weak EG insertion, we need to define a partial order $\leq$ on $R(w)$. 

\begin{definition}\label{incr-fac-def}
    An \defn{increasing factorization} $\rho^{(\bullet)}$ of a reduced word $\rho$
    partitions $\rho$ into consecutive subwords $(\rho^{(k)} | \rho^{(k-1)}|\cdots|\rho^{(1)})$ that are each strictly increasing.
    Some of the subwords $\rho^{(j)}$ in $\rho^{(\bullet)}$ may be empty.
    We say that $\rho^{(\bullet)}$ is a \defn{reduced factorization} for $w$ if $\rho\in R(w)$.
    We denote the $i$-th letter of the $j$-th component in the factorization $\rho^{(\bullet)}$ by $\rho^{(j)}_i$.
    The index $i$ counts from left to right, but the index $j$ counts from right to left,
    following the convention in \cite{Assaf19}.
\end{definition}

There are at least two canonical increasing factorizations associated to a reduced word $\rho$.
First, the \defn{run factorization} of $\rho$
is the increasing factorization obtained by
dividing $\rho$ into increasing subwords of maximal length.
Second, the \defn{trivial factorization} of $\rho$ is the increasing
factorization $(\rho^{(k)} | \rho^{(k-1)}|\cdots|\rho^{(1)})$ with $k= \ell(\rho)$ in which each subword $\rho^{(j)}$ has length one.
For example, the run factorization of $\rho=2736245$
is $(27|36|245)$ 
and the trivial factorization is $(2|7|3|6|2|4|5)$.

\begin{definition}\label{weak descent tableau}
    The \defn{weak descent tableau} of a reduced word $\rho$,
    denoted $\weakdestab(\rho)$,
    is the tableau constructed as follows.
    Suppose $\rho$ has run factorization $(\rho^{(k)} | \rho^{(k-1)}|\cdots|\rho^{(1)})$.
    Place $\rho^{(k)}$ into row $\rho^{(k)}_1$.
    Then iterating over $i=k-1,\dots,2,1$,
    we either place $\rho^{(i)}$ into row $\rho^{(i)}_1$
    if this is below the row containing $\rho^{(i+1)}$
    or place $\rho^{(i)}$ into the row directly below $\rho^{(i+1)}$ otherwise.
\end{definition}

This may result in a tableau with boxes in rows with non-positive index.

\begin{example}\label{weakdestab example}
    Suppose $\rho=2736245$ and $\sigma = 64567342$.
    Then 
    \[
      \weakdestab(\rho)=  \begin{ytableau}\ytableausetup{aligntableaux=center}
            2 & 7\\
            3 & 6 & \none & \none\\
            \hline
            2 & 4 & 5
        \end{ytableau}
        \quand
           \weakdestab(\sigma)=   \begin{ytableau}
            6\\
            \none[\cdot] \\
            4 & 5 & 6 & 7\\
            3 & 4\\
            2\\
            \none[\cdot]
        \end{ytableau}.
    \]
    The horizontal line in $ \weakdestab(\rho)$ divides the positive and non-positive rows.
\end{example}

\begin{remark}
    In \cite{Assaf19}, Assaf defined both a \defn{descent tableau} \cite[Def.~2.3]{Assaf19} and a \defn{weak descent tableau} \cite[Def.~2.8]{Assaf19} of a reduced word $\rho$.
    These are (accidentally) both denoted as $\mathbb{D}(\rho)$.
    Definition~\ref{weak descent tableau} refers to \cite[Def.~2.8]{Assaf19}. In general, the descent tableau and the weak descent tableau of a reduced word are different.
    The descent tableau of $\rho$ can be obtained by removing all empty rows in $\weakdestab(\rho)$.
\end{remark}

A reduced word $\rho$ is \defn{virtual}
if $\weakdestab(\rho)$ occupies a row with a non-positive index.
Otherwise, we say $\rho$ is \defn{non-virtual}.
The \defn{weak descent composition} of a non-virtual reduced word $\rho$ is the shape of $\weakdestab(\rho)$, denoted by $\des(\rho)$.
When $\rho$ is virtual, we define $\des(\rho)=\varnothing$.

\begin{example}
    Continuing from Example~\ref{weakdestab example},
    the reduced word $\rho=2736245$ is virtual and $\des(\rho)=\varnothing$
    since the row consisting of $245$ in the weak descent tableau has a non-positive index.
    But the reduced word $\sigma=64567342$ is non-virtual and $\des(\sigma)=(0,1,2,4,0,1)$.
\end{example}

The \defn{Coxeter--Knuth equivalence relation} is the transitive closure of the relations on reduced words with
$ \mathbf{a}xyz\mathbf{b}      \sim \mathbf{a} xzy \mathbf{b}  $  for $y<x<z$,
with 
$ \mathbf{a}xyz\mathbf{b}      \sim \mathbf{a} yxz\mathbf{b}   $ for $ y<z<x$, and with
$\mathbf{a}i(i+1)i\mathbf{b}  \sim \mathbf{a} (i+1)i(i+1)\mathbf{b}$
whenever $x,y,z,i \in \ZZ_{>0}$ and  $\mathbf{a}$ and $\mathbf{b}$ are possibly empty subwords.

For weak compositions $\mu$ and $\sigma$ of $n$,
we write $\mu\leq \sigma$ if $\sum_{i=1}^k\mu_i\leq \sum_{i=1}^k\sigma_i$ for all $k\geq 1$.
Following \cite{Assaf19}, if $\rho$ and $\tau$ are non-virtual reduced words, then we write $\rho\leq \tau$ if $\rho\sim\tau$ and $\des(\rho)\leq \des(\tau)$ as weak compositions.
This is a partial order,
and each Coxeter--Knuth equivalence class of $R(w)$ contains a unique minimal non-virtual element by \cite[Thm.~4.23]{Assaf19}.
A reduced word is \defn{Yamanouchi} \cite[Def.~4.13]{Assaf19} if it is the minimal element of its Coxeter--Knuth equivalence class under $\leq$. 

\begin{definition}\label{weak EG}
    The \defn{weak Edelman--Greene insertion tableau} $\hat{P}(\rho)$ of a reduced word $\rho$
    is the weak descent tableau of the unique Yamanouchi reduced word $\hat \rho$ with $\rho \sim \hat\rho$.
    We define $\hat P(\rho^{(\bullet)}):=\hat P(\rho)$.
\end{definition}

 Definition~\ref{weak EG} is a generalization of the ordinary \defn{Edelman--Greene insertion} \cite[Def.~6.20]{EG}, which can be defined similarly:
 the \defn{Edelman--Greene insertion tableau} is the unique increasing semistandard Young tableau whose row reading word is Coxeter--Knuth equivalent to $\rho$. 
 Both tableaux can be computed by explicit inductive algorithms, as we explain in the next section.

\subsection{Assaf's lift operation}\label{appendix 1}

Here, we review Assaf's \defn{lift operation} \cite[\S 4.2]{Assaf19}, an algorithm to compute Yamanouchi reduced words.


The lift operation involves the following pairing procedure.
Fix two increasing words $\tau=\tau_1\tau_2\cdots \tau_s$ and $\sigma = \sigma_1\sigma_2\cdots \sigma_t$.
If all letters of $\sigma$ are greater than $\tau_s$,
then the pairing procedure ends and all letters of $\sigma$ are unpaired.
Otherwise, we pair $\tau_s$ with the largest letter $\sigma_{i}$ such that $\tau_s\geq \sigma_{i}$.
Then we repeat the pairing process with the subwords $\tau_1\cdots\tau_{s-1}$ and $\sigma_1\cdots\sigma_{i-1}$.
If the unpaired letters in $\sigma$ are $x_1<\cdots<x_k$, then we can arrange $\tau$ on top of $\sigma$ as
\[
    \begin{array}{llllllll}
        \tau^{(0)} & \tau^{(1)}   &     & \tau^{(2)}   &     & \cdots &       & \tau^{(k+1)}   \\
                   & \sigma^{(1)} & x_1 & \sigma^{(2)} & x_2 & \cdots & x_{k} & \sigma^{(k+1)}
    \end{array}
\]
where $\tau^{(i)}$ and $\sigma^{(i)}$ are possibly empty consecutive subwords with the same length, whose corresponding entries are paired together.

\begin{definition}[{\cite[Def.~4.17]{Assaf19}}]
    If $\tau^{(i)}$, $x_i$, and $\sigma^{(i)}$ are as above then we
    define
    \[\lift(\tau,\sigma)=
        (\tau^{(0)}\tau^{(1)}x_1\tau^{(2)}\cdots x_k\tau^{(k+1)} | {\sigma}^{(1)}\check{\sigma}^{(2)}\cdots\check{\sigma}^{(k+1)}).
    \] 
    For each $1\leq j\leq k$, $\check{\sigma}^{(j+1)}$ denotes the word of length $\ell(\sigma^{(j+1)})$ with
    \[\check{\sigma}^{(j+1)}_i=\begin{cases}
            \sigma^{(j+1)}_i-1 & \text{ for } 1\leq i\leq b_j                       \\
            \sigma^{(j+1)}_i   & \text{ for }  b_j+1\leq i\leq \ell(\sigma^{(j+1)})
        \end{cases}\]
    where
    $b_j\in \ZZ_{\geq 0}$ is maximum such that $\tau^{(j+1)}_i=\sigma^{(j+1)}_i=x_j+i$ for all $1\leq i\leq b_j$.
\end{definition}

An important property of $\lift$ is that if  $\tau$ and $\sigma$ are increasing such that $\tau\sigma$ is a reduced word,
then  $\lift(\tau,\sigma)$ is a reduced word (ignoring the division into factors), and  $\lift(\tau,\sigma)\sim\tau\sigma$ \cite[Lem.~4.18]{Assaf19}.

\begin{definition}[{\cite[Def.~4.19]{Assaf19}}] \label{def 3.1}
    Let $\rho^{(\bullet)} = (\rho^{(k)}| \cdots |\rho^{(1)})$ be an increasing factorization of a reduced word $\rho$.
    Fix $ i \in [k-1]$ and suppose $(\tilde\rho^{(i+1)}|\tilde\rho^{(i)}) = \lift(\rho^{(i+1)},\rho^{(i)})$.
    If $\rho^{(i+1)}$ and $\rho^{(i)}$ are both nonempty and $\tilde\rho^{(i+1)}$ begins with the same letter as $\rho^{(i+1)}$, then we define
    $
        \lift_i(\rho^{(\bullet)}) = (\rho^{(k)}| \cdots | \tilde\rho^{(i+1)}|\tilde\rho^{(i)} |\cdots |\rho^{(1)}).
    $
    Otherwise, let $ \lift_i(\rho^{(\bullet)})  = \rho^{(\bullet)}$.
\end{definition}

For $i\leq j$, we define the \defn{lifting sequence} as $\lift_{[i,j]}=\lift_j\circ\lift_{j-1}\circ \cdots\circ \lift_i$.
    We say $\lift_i$ acts \defn{faithfully} on an increasing factorization $\rho^{(\bullet)}$ if $\lift_i(\rho^{(\bullet)})\neq \rho^{(\bullet)}$.
    We say $\lift_{[i,j]}$ acts faithfully if $\lift_i$ acts \defn{faithfully} on $\rho^{(\bullet)}$ and $\lift_k$ acts faithfully on $\lift_{[i,k-1]}(\rho^{(\bullet)})$ for $i<k\leq j$.

\begin{definition} \label{def 3.3}
    For an increasing factorization $\rho^{(\bullet)}=(\rho^{(k)}|\cdots |\rho^{(1)})$,
    we construct the increasing factorization $\lift(\rho^{(\bullet)})$ as follows.
    Set $(\sigma_0)^{(\bullet)}=\rho^{(\bullet)}$, and assume $(\sigma_1)^{(\bullet)},\dots,(\sigma_{n-1})^{(\bullet)}$ are known.
    \begin{enumerate}[label=(\arabic{*})]
        \item If $\lift_i\bigl((\sigma_{n-1})^{(\bullet)}\bigr)=(\sigma_{n-1})^{(\bullet)}$ for all $i$, then $\lift(\rho^{(\bullet)})=(\sigma_{n-1})^{(\bullet)}$.
        
        \item Otherwise, set $(\sigma_n)^{(\bullet)}=\lift_{[i_n,j_n]}\bigl((\sigma_{n-1})^{(\bullet)}\bigr)$ 
        where $j_n$ is the maximum $j<k$ for which there exists $i\leq j$ such that $\lift_{[i,j]}$ acts faithfully on $(\sigma_{n-1})^{(\bullet)}$, and $i_n$ is the minimum $i\leq j_n$ for which $\lift_{[i,j_n]}$ acts faithfully on $(\sigma_{n-1})^{(\bullet)}$.
    \end{enumerate}
\end{definition}

\begin{definition}\label{lift(T)}
    For an increasing Young tableau $T$ with reduced reading word, let $\rho^{(\bullet)}$ be the run factorization of $\rowreading(T)$ and define $\lift(T):=\weakdestab(\lift(\rho^{(\bullet)}))$.
\end{definition}

\begin{remark}\label{lift-remark}
If $\lift(\rho^{(\bullet)}) = (\eta^{(k)}| \eta^{(k-1)}|\cdots|\eta^{(1)})$
then
$\lift(T)$ is obtained by just placing the word $\eta^{(i)}$ in row $\eta^{(i)}_1$.
    Definition~\ref{lift(T)} is intended to match \cite[Def.~4.22]{Assaf19}, which uses the same notation $\lift(T)$.
    However,  Assaf's definition in \cite{Assaf19} specifies $\lift(T)$ as the tableau obtained by  placing $\eta^{(i)}$ in row $i$.
    This appears to be a mistake; although \cite[Fig.~17]{Assaf19} matches  \cite[Def.~4.22]{Assaf19}, the later figure \cite[Fig.~21]{Assaf19} and associated results use Definition~\ref{lift(T)}. For example, see Figure~\ref{lift-figure}.
\end{remark}

\begin{figure}[ht]
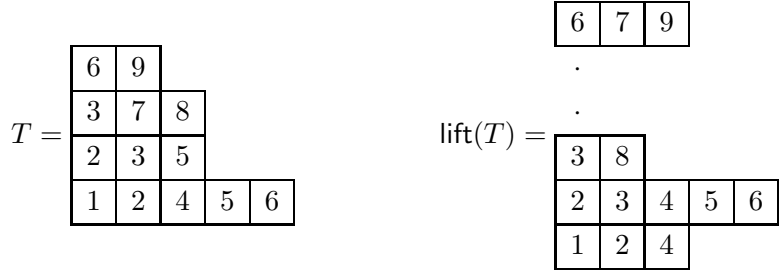

    \begin{displaymath}
        T =\begin{ytableau}\ytableausetup{aligntableaux=center}
            6 & 9 \\
            3 & 7 & 8\\
            2 & 3 & 5 \\
            1 & 2 & 4 & 5 & 6
        \end{ytableau}
        \hspace{5em}
        \lift(T)= \begin{ytableau} \ytableausetup{aligntableaux=center}
            6 & 7 &  9 \\
            \none[\cdot]\\
            \none[\cdot] \\
            3 & 8 \\
            2 & 3 & 4 & 5 & 6 \\
            1 & 2 & 4
        \end{ytableau}
    \end{displaymath}
    \caption{\label{lift-figure}The run factorization of $\row(T)$ is $(69|378|235|12456)$, and $T$ has lift sequences $[i_0,j_0]=[2,3]$ and $[i_1,j_1]=[1,1]$.}
\end{figure}


For a reduced word $\rho$, let $P(\rho)$ be the EG insertion tableau of $\rho$; see \cite[Def.~6.20]{EG}. 
\begin{theorem}[{\cite[Thm.~4.23]{Assaf19}}] 
\label{thm 3.20}
If $\rho$ is a reduced word then $\hat{P}(\rho)=\lift(P(\rho))$.
\end{theorem}

We note two useful properties of $P(\rho)$. 
A reduced word $\rho$ has a \defn{descent} at $i$ if $\rho_i>\rho_{i+1}$ where $1\leq i\leq \ell(\rho)-1$. 
A standard Young tableau has a \defn{descent} at $i$ if $i+1$ appears in one of the rows above $i$.
Edelman and Greene proved that the EG insertion preserves descents \cite[Lem.~6.28]{EG}. This property combined with Theorem~\ref{thm 3.20} implies:

\begin{corollary}\label{corollary 331}
    Suppose $\rho$ and $\rho xy$ are reduced for $x,y\in \ZZ_{>0}$.
    Suppose the box in $\hat{P}(\rho x)\setminus\hat{P}(\rho)$ is in column $c_x$
    and the box in $\hat{P}(\rho xy)\setminus \hat{P}(\rho x)$ is in column $c_y$.
    Then $x<y$ if and only if $c_x < c_y$.
\end{corollary}

\begin{figure}[ht]
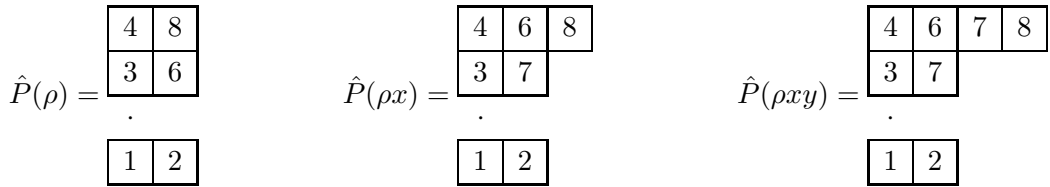

    \begin{displaymath}
        \hat P(\rho )=\begin{ytableau} \ytableausetup{aligntableaux=center}
            4 & 8 \\
            3 & 6 \\
            \none[\cdot] \\
            1 & 2 
        \end{ytableau}
            \hspace{5em}\hat P(\rho x)=\begin{ytableau} \ytableausetup{aligntableaux=center}
                4 & 6&8 \\
                3 & 7 \\
                \none[\cdot] \\
                1 & 2 
            \end{ytableau}
            \hspace{5em}\hat P(\rho xy)=\begin{ytableau} \ytableausetup{aligntableaux=center}
                4 & 6&7 & 8\\
                3 & 7 \\
                \none[\cdot] \\
                1 & 2 
            \end{ytableau}
        \end{displaymath}
        \caption{Let $\rho= 438612$, $x=7$, and $y=8$. Then $c_x = 3$ and $c_y =4$.}
    \end{figure}
It also holds that if $i$ the smallest letter in $\rho$ then $P(\rho)_{(1,1)}=i$. Therefore:

\begin{corollary}\label{smallest row index}
    If $i$ is the smallest letter in $\rho$ then the first nonempty row in $\hat P(\rho)$ has index $i$.
\end{corollary}

\subsection{Weak EG recording tableaux}

This section introduces a recording tableau $\hat Q(\rho^{(\bullet)})$ for the weak EG insertion 
which slightly generalizes constructions in \cite{Assaf19, AS}. These references 
defined $\hat Q(\rho^{(\bullet)})$ when $\rho^{(\bullet)}$ is either the trivial factorization or bounded by the standard flag.
Here, we extend to arbitrary flagged factorizations. 
    
\begin{definition}
    The \defn{weak Edelman--Greene recording tableau} $\hat{Q}(\rho^{(\bullet)})$ is the tableau
    with same shape as $\hat{P}(\rho^{(\bullet)})$ having $i$ in the all boxes
    that are in $\hat{P}(\rho^{(k)} |\cdots|\rho^{(i+1)}|\rho^{(i)})$
    but not in $\hat{P}(\rho^{(k)} |\cdots|\rho^{(i+1)})$.
    We define $\hat{Q}(\rho)$ to be the weak EG recording tableau of the trivial factorization of $\rho$.
\end{definition}
The weak EG recording tableau is well-defined because if $\rho$ and $\rho x$ are reduced for $x\in \ZZ_{>0}$
then $\hat{P}(\rho x)\setminus \hat P(\rho)$ has exactly one box \cite[Lem.~5.8]{Assaf19}. 
\begin{example}
    Let $\rho^{(\bullet)}=(3|26|56|4)$.
    As we insert $\rho^{(\bullet)}$ from left to right, insertion tableaux are
    \[
    \begin{array}{cccccc}
        \ytableausetup{aligntableaux=bottom}
            \begin{ytableau}3\\\none[\cdot]\\\none[\cdot]\end{ytableau}
             & \begin{ytableau}3\\2\\\none[\cdot]\end{ytableau}
             & \begin{ytableau}3& 6\\2\\\none[\cdot]\end{ytableau}
             & \begin{ytableau}3& 6\\2&5\\\none[\cdot]&\none[\cdot]\end{ytableau}
             & \begin{ytableau}3& 5&6\\2&5\\\none[\cdot]&\none[\cdot]\end{ytableau}
             & \begin{ytableau}6\\\none[\cdot]\\\none[\cdot]\\3 & 5& 6 \\2&  4\\\none[\cdot]&\none[\cdot]\end{ytableau}
\end{array}.\]
The corresponding weak EG recording tableau $\hat Q$ at each step is as follows:
    \[
    \begin{array}{cccccc}
\begin{ytableau}4\\\none[\cdot]\\\none[\cdot]\end{ytableau}
& \begin{ytableau}4\\3\\\none[\cdot]\end{ytableau}
& \begin{ytableau}4& 3\\3\\\none[\cdot]\end{ytableau}
& \begin{ytableau}4& 3\\3&2\\\none[\cdot]&\none[\cdot]\end{ytableau}
             & \begin{ytableau}4 &3 &2\\ 3 & 2\\\none[\cdot]&\none[\cdot]\end{ytableau}
             & \begin{ytableau}1\\\none[\cdot]\\\none[\cdot]\\4& 3& 2 \\3& 2 \\\none[\cdot]&\none[\cdot]\end{ytableau}
        \end{array}.\]
        The final tableaux on the right are $\hat P(\rho^{(\bullet)})$ and $\hat Q(\rho^{(\bullet)})$.
\end{example}

Let $\beta=(\beta_1,\beta_2,\dots)$ be the weight of a key tableau $T$. 
We \defn{standardize} $T$ by the following procedure:
first replace all the $1$'s by $1,2,\dots, \beta_1$ from right to left, then replace all the $2$'s by $\beta_1+1,\beta_1+2,\dots, \beta_1+\beta_2$, and so on. 
Denote the result by $\stdk\big(T\big)$,
which is a standard key tableau of the same shape as $T$.
If $\rho$ is a reduced word, then
$\stdk\big(\hat{Q}(\rho^{(\bullet)})\big)$ is equal to $\hat{Q}(\rho)$.
If \[
T=\begin{ytableau}
        \ytableausetup{aligntableaux=center}
        3\\\none[\cdot]\\2& 2 \\5& 4& 3& 1 \\\none[\cdot]&\none[\cdot]&\none[\cdot]&\none[\cdot]\end{ytableau}
        \quad\text{then}\quad
        \stdk(T)=\begin{ytableau}5\\\none[\cdot]\\3& 2 \\7& 6& 4& 1 \\\none[\cdot]&\none[\cdot]&\none[\cdot]&\none[\cdot]\end{ytableau}.\]

\begin{lemma}
    The weak EG recording tableau $\hat{Q}(\rho^{(\bullet)})$ is a key tableau.
\end{lemma}
\begin{proof}
    By \cite[Thm.~5.9]{Assaf19}, $\stdk\big(\hat{Q}(\rho^{(\bullet)})\big)$ is a standard key tableau. 
The conditions in Definition~\ref{key tableau def} follow by the definition of standardization.
\end{proof}
 
Fix a flag $\varphi$, that is, a weakly increasing map $\varphi:\ZZ_{> 0}\to \ZZ_{> 0}$ with $\varphi(i)\geq i$.

\begin{definition}\label{rfc-def}
We say that $\rho^{(\bullet)}$ is \defn{$\varphi$-flagged}
    if for each $i\geq 1$, the first entry of block $i$ satisfies $\varphi(\rho_1^{(i)})\geq i$ when $\rho^{(i)}$ is nonempty.
    We denote the set of all $\varphi$-flagged reduced factorizations for $w$ as $\rfc(w,\varphi)$.
    If $\varphi$ is the standard flag, then we omit $\varphi$ in the notation and write $\rfc(w)$, which is the same as the set of \defn{reduced factorizations with cutoff} defined in \cite[Def.~5.4]{AS}.
\end{definition}

\begin{remark}\label{remark 318}
    Fix $n\in \ZZ_{>0}$ and let $\rfc_n(w,\varphi)$ be the subset of $\rfc(w,\varphi)$
    consisting of factorizations whose nonempty components lie in the first $n$ components.
    Let $t=\min \{i\in \ZZ_{>0}\mid \varphi(i)\geq n\}$.
    If $\phi$ is another flag such that $\varphi(j)=\phi(j)$ for all $1\leq j\leq t$,
    then $\rfc_n(w,\varphi)=\rfc_n(w,\phi)$.
\end{remark}

\begin{lemma}\label{flag lemma}
    The weak EG recording tableau $\hat{Q}(\rho^{(\bullet)})$ is
    $\varphi$-flagged (in the sense that any entry in the $i$-th row is at most $\varphi(i)$) if and only if $\rho^{(\bullet)}$ is $\varphi$-flagged.
\end{lemma}

\begin{proof}

    Suppose $\hat{Q}(\rho^{(\bullet)})$ is $\varphi$-flagged. 
    Let $\hat{Q}(\rho^{(\bullet)})|_{[i,k]}$ be the restriction to boxes with entries in $\{i,\dots,k\}$. 
    For $1\leq i \leq k$, $\hat{Q}(\rho^{(\bullet)})|_{[i,k]}$ is $\varphi$-flagged.
    Assume $\rho^{(i)}$ is nonempty, and 
    suppose the first nonempty row of $\hat{Q}(\rho^{(\bullet)})|_{[i,k]}$ is row $l_i$ with the maximum entry $j$. 
    By the flagged condition, we have $i\leq j\leq \varphi(l_i)$.
    Because the minimal letter of $\rho^{(k)}\rho^{(k-1)}\cdots\rho^{(i)}$ is $l_i$ by Corollary~\ref{smallest row index}, the first entry $\rho^{(i)}_1$ of the $i$-th component is at least $l_i$. Hence, $\varphi(\rho^{(i)}_1)\geq\varphi(l_i)\geq i$. 
    Thus, $\rho^{(\bullet)}$ is $\varphi$-flagged.

    Conversely, suppose $\rho^{(\bullet)}$ is $\varphi$-flagged.
    The letters in $\rho^{(k)}\cdots \rho^{(i)}$ are bounded below by $l_i$,
    where $l_i=\min \{n\mid \varphi(n)\geq i\}$.
    By Corollary~\ref{smallest row index}, $\hat{Q}(\rho^{(\bullet)})|_{[i,k]}$ has no boxes below row $l_i$.
    If an $i$-entry in $\hat{Q}(\rho^{(\bullet)})$ shows up in row $j$, then we have $ l_i\leq j$ and $i\leq \varphi(l_i)\leq \varphi(j)$.
    Hence, $\hat{Q}(\rho^{(\bullet)})$ is $\varphi$-flagged.
\end{proof}

Given a weak composition $\alpha$,
let $\YR_\alpha(w)$ be the set of Yamanouchi reduced words $\sigma$ for $w$ such that $\des(\sigma)=\alpha$. The map $\weakdestab(-):\YR_\alpha(w)\to \{\hat P(\rho)\mid \rho\in R(w),\sh\(\hat P(\rho)\)=\alpha\} $ is a bijection and the inverse map is given by $\rowreading(-)$.
Here we generalize \cite[Cor.~5.8]{AS} to all flags.

\begin{proposition}\label{flagged bijection}
    The weak EG insertion map $\rho^{(\bullet)}\mapsto \big(\hat{P}(\rho^{(\bullet)}),\hat{Q}(\rho^{(\bullet)})\big)$ is a weight-preserving bijection
    $\rfc(w,\varphi)\to \bigsqcup_{\alpha}\big(\YR_\alpha(w)\times \SSKT(\alpha,\varphi)\big)
    $.
\end{proposition}

\begin{proof}

    Suppose $\YR_\alpha(w)$ is nonempty and $(\hat{P},\hat{Q})\in\YR_\alpha(w)\times \SSKT(\alpha,\varphi)$.
    There exists a unique reduced word $\rho$ 
    such that $\hat{P}(\rho)=\hat{P}$
    and $\hat{Q}(\rho)=\stdk(\hat{Q})$ by {\cite[Cor.~5.12]{Assaf19}}.
    Let $\beta=(\beta_1,\dots,\beta_k)$ be the weight of $\hat{Q}$ and define $\rho^{(\bullet)}=(\rho^{(k)} |\cdots|\rho^{(1)})$ to be the unique factorization
    of $\rho $ with  
    $\ell(\rho^{(i)})=\beta_i$. 
    The factorization is increasing by Corollary~\ref{corollary 331} and by construction $\hat{Q}(\rho^{(\bullet)})=\hat{Q}$.
Since  $\hat{Q}$ is $\varphi$-flagged, Lemma~\ref{flag lemma} implies that $\rho^{(\bullet)}$ is $\varphi$-flagged.
    Hence, the weak EG insertion map is surjective.
    
   Suppose $\sigma^{(\bullet)}\in \rfc(w,\varphi)$ such that $(\hat{P}(\rho^{(\bullet)}),\hat{Q}(\rho^{(\bullet)}))=(\hat{P}(\sigma^{(\bullet)}),\hat{Q}(\sigma^{(\bullet)}))$.
    Then $\stdk(\hat{Q}(\rho^{(\bullet)}))=\stdk(\hat{Q}(\sigma^{(\bullet)}))$, so 
     $\rho =\sigma$ by \cite[Cor.~5.12]{Assaf19}.
    Hence $\rho^{(\bullet)} =\sigma^{(\bullet)}$ 
    because $\weight(\hat{Q}(\rho^{(\bullet)})) = \weight(\hat{Q}(\sigma^{(\bullet)}))$.
\end{proof}

\section{Crystal structures}\label{application}

In this section, we prove our main results from the introduction.
Throughout, we fix $n\in \ZZ_{>0}$, $w\in  S_\infty$, a flag $\varphi$, and a weak composition $\alpha$. 
For a positive integer $k$, let $[k]:=\{1,2,\dots,k\}$.

\subsection{Crystal structure on reduced factorizations}\label{4.1}
We begin with the basics of \defn{Morse--Schilling crystals}  on reduced factorizations \cite[\S 3.2]{MS}.
All \defn{crystals} refer to \defn{$\gl_{n}$-crystals} in the sense of \cite{BS}. 
Such a crystal consists of a finite set $\cB$ with \defn{raising} and \defn{lowering operators} $e_i,f_i : \cB\to \cB \sqcup\{0\}$ indexed by $1\leq i \leq n-1$, 
along with a weight function $\weight$ taking values in $\ZZ^n$.
It is required that for any $b,c\in \cB$, we have $e_i(b)=c$ if and only if $f_i(c)=b$, and in this case $\weight(c)= \weight(b)+\e_{i}-\e_{i+1}$ where $\e_i$ is the $i$-th standard basis vector of $\ZZ^n$.
Literature on crystals sometimes involves additional axioms, but we will not impose any of those here.

The \defn{crystal graph} of $\cB$ is a directed graph with vertices in $\cB$ and edges labeled by $[n-1]$. 
For $x,y\in \cB$, we draw an edge $x\xrightarrow{i}y$ if $f_i(x)=y$. 
A \defn{connected component} of $\cB$ is a subset of $\cB$ whose elements form a connected component in the crystal graph of $\cB$. 
The \defn{character} of a finite crystal $\cB$ is the Laurent polynomial $\ch(\cB):=\sum_{b\in \cB} x^{\weight(b)}$.
A \defn{crystal isomorphism} between crystals $\cB$ and $\cC$ is a weight-preserving bijection $\cB\to \cC$ that commutes with all raising and lowering operators.

Let $\rf_n(w)$ be the set of all reduced factorizations $r^{(\bullet)}=(r^{(n)}|\cdots |r^{(1)})$ for
$w\in  S_\infty$ (as specified in Definition~\ref{incr-fac-def}) 
with exactly $n$ components, some of which may be empty.
Suppose $r^{(\bullet)}=(r^{(n)}|\cdots |r^{(1)})\in \rf_n(w)$.
The operators $e_i$ and $f_i$ applied to $r^{(\bullet)}$ only change the factors $r^{(i+1)}$ and $r^{(i)}$.
The definition of these operators depends on the following pairing procedure.

Starting with the largest element $b$ in $r^{(i)}$,
pair it with the smallest element $a$ in $r^{(i+1)}$ with $a>b$.
If there is no such $a$ then $b$ is unpaired.
Next, we pair the second-largest element $b'$ in $r^{(i)}$ with the smallest unpaired element $a'$ in  $r^{(i+1)}$ with $a'>b'$.
If there is no such $a'$ then $b'$ is unpaired.
We continue this procedure for the remaining elements of $r^{(i)}$ in decreasing order, ignoring at each stage
any elements in $r^{(i+1)}$ that have already been paired.
Once the procedure ends,
we define 
\be
\begin{aligned}
R_i(r^{(\bullet)}) &= \{b\in r^{(i)}\mid b \text{ is unpaired in the pairing of }r^{(i+1)}r^{(i)}\}, \\
L_i(r^{(\bullet)}) & =\{a\in r^{(i+1)}\mid a \text{ is unpaired in the pairing of }r^{(i+1)}r^{(i)}\}.
\end{aligned}
\ee
Then $f_i(r^{(\bullet)})$ and $e_i(r^{(\bullet)})$ are given as follows:

\begin{definition}
If $R_i(r^{(\bullet)})=\varnothing$
then $f_i(r^{(\bullet)})=0$.
Otherwise, $f_i(r^{(\bullet)})$ is obtained by replacing $r^{(i+1)}$
and $ r^{(i)}$ with $\tilde{r}^{(i+1)}$ and $\tilde{r}^{(i)}$, respectively,
where
\[\tilde{r}^{(i)}=r^{(i)}\setminus \{b\} \quad\text{and}\quad  \tilde{r}^{(i+1)}=r^{(i+1)}\cup\{b-t\}\]
for $b=\min R_i(r^{(\bullet)})$ and $t=\min \{j\geq 0\mid b-j-1\notin r^{(i)}\}$.
Similarly, if $L_i(r^{(\bullet)})=\varnothing$
then $e_i(r^{(\bullet)})=0$.
Otherwise, $e_i(r^{(\bullet)})$ is obtained by replacing $r^{(i+1)}$ and $ r^{(i)}$ with $\tilde{r}^{(i+1)}$ and $\tilde{r}^{(i)}$, respectively,
where
\[\tilde{r}^{(i)}=r^{(i)}\cup \{a+s\} \quad\text{and}\quad  \tilde{r}^{(i+1)}=r^{(i+1)}\setminus\{a\}\]
for $a=\max L_i(r^{(\bullet)})$ and $s=\min \{j\geq 0\mid a+j+1\notin r^{(i+1)}\}$.
\end{definition}

The pairing procedure ensures that $b-t+1,\dots, b-1,b \in r^{(i+1)}$ and $a,a+1,\dots,a+s-1\in r^{(i)}$. 
Since $r$ is reduced, we have $b-t\notin r^{(i+1)}$ and $a+s\notin r^{(i)}$. 

\begin{example}
    Let $r^{(\bullet)}=(268|14|345)$.
    Then   $L_1(r^{(\bullet)})=\{1\}$ and $R_1(r^{(\bullet)})=\{5,4\}$,
    so  $f_1(r^{(\bullet)})=(268|134|35)$ and $e_1(r^{(\bullet)})=(268|4|1345)$.
\end{example}

For $r^{(\bullet)}\in\rf_n(w)$, let $\weight(r^{(\bullet)}) := (\ell_1,\ell_2,\dots,\ell_n)$
where $\ell_i$ is the length of $r^{(i)}$. For example, this gives $\weight((268|134|35)) = (2,3,3)$.
The operators $f_i$ and $e_i$ for $1 \leq i < n$ and the weight function $\weight$ define a $\gl_{n}$-crystal structure on $\rf_n(w)$ \cite[Thm.~3.5]{MS}. Additionally, if $e_i(r^{(\bullet)})\neq 0$
then the underlying reduced words of $r^{(\bullet)}$ and $e_i(r^{(\bullet)})$ are Coxeter--Knuth equivalent by \cite[Thm.~4.11]{MS}.

\begin{lemma}\label{crystal and flag}
If $r^{(\bullet)}\in\rf_n(w)$ is $\varphi$-flagged then $e_i(r^{(\bullet)})$ is $\varphi$-flagged or zero.
\end{lemma}

\begin{proof}
Assume $r^{(\bullet)}$ is $\varphi$-flagged. Each number $i$ only appears in the rightmost $\varphi(i)$ components.
Since $e_i(r^{(\bullet)}) \neq 0$ is formed by removing $a$ from $r^{(i+1)}$ and adding $a+s\geq a$ to $r^{(i)}$, 
it is $\varphi$-flagged.
\end{proof}

\subsection{Demazure crystals}\label{s4.2}

Recall that if $\cB$ is a crystal and $X\subseteq \cB$, then we define
    $\fkD_i^{\cB}X:=\{b\in \cB\mid e_i^k(b)\in X \text{ for some }k\geq 0\}$.
    We abbreviate the \defn{Demazure operator} $\fkD_i^{\cB}$ as $\fkD_i$ if $\cB$ is clear from the context. 

Suppose $u^{(\bullet)}\in \rf_n(w)$ is a \defn{highest weight element} in the sense that $e_i(u^{(\bullet)})=0$ for all $1\leq i<n$.  
If $i_1\cdots i_k,j_1\cdots j_k\in R(\sigma)$ for some $\sigma\in S_n$ and $\fkD_i:=\fkD_i^{\rf_n(w)}$, 
then $\fkD_{i_1}\cdots\fkD_{i_k}\{u^{(\bullet)}\}=\fkD_{j_1}\cdots\fkD_{j_k}\{u^{(\bullet)}\}$ \cite[Thm.~13.5]{BS}.
We can therefore define $\fkD_\sigma\{u^{(\bullet)}\}:=\fkD_{i_1}\cdots\fkD_{i_k}\{u^{(\bullet)}\}$ 
when $i_1\cdots i_k\in R(\sigma)$ and refer to $\fkD_\sigma\{u^{(\bullet)}\}$ as a \defn{Demazure subcrystal} of $\rf_n(w)$. 
We view $\fkD_\sigma\{u^{(\bullet)}\}$ as a crystal by redefining $e_i$ and $f_i$ to act as zero whenever they would send an element outside the subset.
The Demazure character formula \cite[Thm.~13.7]{BS} implies that $\ch(\fkD_\sigma\{u^{(\bullet)}\})=\pi_\sigma x^{\weight(u^{(\bullet)})}$.

\begin{definition}\label{dem-def}
A \defn{Demazure crystal} is a crystal isomorphic to a Demazure subcrystal of $\rf_n(w)$ for some $w \in S_\infty$ and some $n\in \ZZ_{>0}$.
\end{definition}

It is known (see \cite{K}) that 
two Demazure crystals are isomorphic if and only if they have the same character, and that every $\kappa_\alpha$ (with $\ell(\alpha)\leq n$) occurs as the character of some Demazure crystal.
Moreover, any Demazure crystal that can be embedded in $\rf_n(w)$ must be equal to $\fkD_\sigma\{u^{(\bullet)}\}$ for some $\sigma \in S_n$ and some highest weight element $u^{(\bullet)}$.

In \cite[Thm.~5.11]{AS}, Assaf and Schilling showed that $\rfc(w)$ is a union of Demazure crystals.
If $\varphi$ is a flag, we denote $\varphi-\e_i:\ZZ_{>0}\to \ZZ_{\geq 0}$ to be the function such that $(\varphi-\e_i)(l)=\varphi(l)-\delta_{il}$, where $\delta_{il}$ is the Kronecker delta. 
The following theorem is the main ingredient to prove Theorem~\ref{theorem2}.
\begin{theorem}\label{thm 44}
    Suppose $\varphi$ is non-standard, and $i\in\ZZ_{>0}$ is minimal with $\varphi(i)>i$.
    Let $j=\varphi(i)-1$. 
    If $j\geq n$ then $\rfc_n(w,\varphi)=\rfc_n(w,\varphi-\e_i)$.
    Otherwise, $\rfc_n(w,\varphi)=\fkD^{\rf_n(w)}_j\rfc_n(w,{\varphi-\e_i})$. 
\end{theorem}
\begin{proof}
    The flagged condition says that $i$ can only show up in the first $j+1=\varphi(i)$ blocks of any factorization in $\rfc_n(w,\varphi)$,  counting from right to left.
  However, $i$ can only show up in the first $j$ blocks for any factorization in $\rfc_n(w,\varphi-\e_i)$.
    Therefore, we have $\rfc_n(w,{\varphi-\e_i})\subseteq\rfc_n(w,{\varphi})$.
    When $j\geq n $, the desired equality holds by Remark~\ref{remark 318}. Assume $j<n$ from now on.

    Suppose $r^{(\bullet)}\in \rfc_n(w,\varphi)\setminus \rfc_n(w,{\varphi-\e_i})$.
   Then  $r^{(j+1)}$ must start with $i$,
    but $r^{(j)}$ does not contain any number smaller than $i$ since
    \[\varphi(i-1)=i-1<i\leq \varphi(i)-1=j.\]
    Therefore, we will have $i\in L_j(r^{(\bullet)})$ when we pair $r^{(j+1)}$ and $r^{(j)}$.
    The operator $e_j$ removes the largest element $a\in L_j(r^{(\bullet)})$ from $r^{(j+1)} $ and adds $a+s$ to $r^{(j)}$,
    where $s$ is the smallest non-negative integer such that $a+s+1\notin r^{(j+1)}$.
    If $a=i$, then block $j+1$ of $e_j(r^{(\bullet)})$ does not contain $i$ and $e_j(r^{(\bullet)})\in\rfc_n(w,{\varphi-\e_i})$.
    If $a\neq i$, then we have $a+s\geq a>i$ so $e_j(r^{(\bullet)})\in \rfc_n(w,\varphi)\setminus \rfc_n(w,{\varphi-\e_i})$ and $i\in L_j\(e_j(r^{(\bullet)})\)$.
   Since we can only apply $e_j$ a finite number of times before reaching zero,
   we must have $e_j^k(r^{(\bullet)}) \in \rfc_n(w,{\varphi-\e_i})$ for some $k\geq 0$, so $\rfc_n(w,\varphi)\subseteq\fkD_j \rfc_n(w,{\varphi-\e_i})$.

    Conversely, suppose $r^{(\bullet)}\in \fkD_j \rfc_n(w,{\varphi-\e_i})$.
    We want to show that $r^{(\bullet)}$ is $\varphi$-flagged.
    By definition, $e_j^{k}(r^{(\bullet)})\in \rfc_n(w,{\varphi-\e_i})$ for some $k\geq 0$,
    which is the same as saying $f_j^k (u^{(\bullet)})=r^{(\bullet)}$ for some $u^{(\bullet)}\in \rfc_n(w,\varphi-\e_i)$.
    Fix $v^{(\bullet)}\in \rfc_n(w,{\varphi})$ such that $f_j(v^{(\bullet)})\neq 0$.
    Since $\rfc_n(w,{\varphi-\e_i})\subseteq\rfc_n(w,{\varphi})$,
    it suffices to show that $f_j(v^{(\bullet)})\in \rfc_n(w,{\varphi})$.

    The factorization $f_j(v^{(\bullet)})$ is obtained from $v^{(\bullet)}$ by removing $b=\min R_j(v^{(\bullet)})$ from block $j$ and adding $b-t$ to block $j+1$,
    where $t$ is minimal such that $b-t-1\notin v^{(j)}$.
    Since $v^{(\bullet)}$ is $\varphi$-flagged, 
    $f_j(v^{(\bullet)})$ can fail to be $\varphi$-flagged only if $\varphi(b-t) < j + 1 = \varphi(i)$,
    which can only happen if $b-t < i$.
    However, as 
    $i-1$ can only show up in the first $i-1 = \varphi(i-1)$ blocks of $v^{(\bullet)}$, and 
   as  $j=\varphi(i)-1>i-1$, we must have $i-1 \notin v^{(j)}$ so $b-t-1 \geq i-1$. Thus 
      $b-t\geq i$ and $f_j(v^{(\bullet)})\in \rfc_n(w,{\varphi})$.
\end{proof}

For $n \geq b >a \geq 1$, let $\fkD_{b\downarrow a}:=\fkD_{b-1}^{\rf_n(w)} \fkD_{b-2}^{\rf_n(w)} \cdots\fkD_{a}^{\rf_n(w)}  $
and $\fkD_{a\downarrow a}:=\id$.

\begin{corollary}\label{corollary 4.5}
    Let $t_i=\min\{n,\varphi(i)\}$ for $1\leq i \leq n$.
    Then we have
    \[
        \rfc_n(w,\varphi)=\fkD_{t_1\downarrow 1}\fkD_{t_2\downarrow 2}\cdots\fkD_{t_n\downarrow n}\rfc_n(w).
    \]
\end{corollary}


\begin{proof}
    Let $\overline\varphi$ be the modified flag such that $\overline\varphi(i) = \min\{n,\varphi(i)\}$ for all $1\leq i \leq n$. Since $\rfc_n(w,\varphi) = \rfc_n(w,\overline\varphi)$, the result follows by repeatedly applying Theorem~\ref{thm 44}.
\end{proof}
\subsection{Crystal structures on key tableaux}\label{4.2}

Assaf and Schilling \cite[\S 3.2]{AS} defined a $\gl_n$-crystal on key tableaux, and they proved that $\SSKT(\alpha)$ is a Demazure crystal. Instead of reviewing their definition, we summarize some important properties.

First, the raising operator $e_i$ specified in \cite[Def.~3.7]{AS} can be applied to all key tableaux $T$ not just the semistandard ones. If $T$ is a key tableau, then $e_i(T)$ is obtained by changing some entries $i+1$ in the same row of $T$ to $i$ and then changing all $i$'s in the same columns as these entries to $i+1$'s. 
All entries $i+1$ changed by $e_i$ are in consecutive columns, and each of these entries has an $i$ above it except for the rightmost one; see the proof of  \cite[Prop.~3.8]{AS}.

\begin{proposition}
    If $T\in \SSKT(\alpha,\varphi)$
    then $e_i(T)\in \SSKT(\alpha,\varphi)\sqcup\{0\} $. 
\end{proposition}

\begin{proof}
    
    Assume $e_i(T)\neq 0$ and those $i+1$ changed by $e_i$ are in row $r$.
    Then $i+1\leq \varphi(r)$ by the flagged condition.
    Replacing an $i+1$ by $i$ does not violate the flagged condition.
    For an entry $i$ in row $r'$ replaced by $i+1$, we have $r'>r$ by the observation in the previous paragraph.
    Hence, we have $i+1\leq \varphi(r)\leq \varphi(r')$, so $e_i(T)\in \SSKT(\alpha,\varphi)$.
\end{proof}

There is a unique way to define \defn{lowering operators} $f_i : \SSKT(\alpha,\varphi) \to \SSKT(\alpha,\varphi) \sqcup \{0\}$
such that $e_i(T) = U$ if and only if $T = f_i(U)$ for $T,U \in \SSKT(\alpha,\varphi)$; 
see \cite[Def.~3.10]{AS}, which also applies in our $\varphi$-flagged case.
By the previous proposition, we can view the set $\SSKT(\alpha,\varphi)$ as a $\gl_n$-crystal with raising operators $e_i$ for all $n$ with
$\varphi(\ell(\alpha))\leq n$. 
The bound on $n$ is necessary and sufficient for the weight function of $\SSKT(\alpha,\varphi)$ to take values in $\ZZ^n$. 
Let $\lambda$ be the partition rearrangement of $\alpha$ and $\lambda^T$ be the transpose of $\lambda$. Inspecting \cite[Def.~3.7]{AS} gives the following proposition:
\begin{proposition}\label{highest-prop}
    The unique highest weight element in $\SSKT(\alpha,\varphi)$ is the tableau in which column $i$ is filled by $1,2,\dots,\lambda^T_i$ from bottom to top. 
\end{proposition}
    

Assaf and Schilling showed that the crystal operators   
for $\rfc(w)$ and $\SSKT(\alpha)$
commute with the weak EG recording tableau $\hat{Q}(-)$ \cite[Thm.~5.10]{AS}.
Our goal is to show that this relation remains true for all flags.
Before we proceed to the proof, we need to define two shifting maps:
one on $\rfc(w,\varphi)$, and the other on $\SSKT(\alpha,\varphi)$.
These shifting maps commute with the crystal operators $e_i$.
Hence, the action of $e_i$ on $\rfc(w,\varphi)$ and $\SSKT(\alpha,\varphi)$ can be computed from
the action of $e_i$ in the case of the standard flag. 
 
Let $N$ be a positive integer.
If $w\in S_m$ 
then $1_N \times w \in S_{N+m}$ is the permutation fixing $[N]$ that has $w(i+N)=w(i)+N$ for all $i\in \ZZ_{>0}$.
We define $\shift_N:\rf_n(w)\to\rf_n(1_N\times w)$ by adding $N$ to each letter of every factor.
Since the flagged condition on reduced factorizations gives a lower bound for the letters in each component,
the shifting map $\shift_N$ sends $\rfc(w,\varphi)\hookrightarrow \rfc(1_N\times w,\varphi)$.

For any key tableau $T$,
we define $\shift_N(T)$ to be the key tableau obtained from $T$ by shifting up $N$ rows.
By the definition, $\shift_N$ preserves the flagged condition. 
Then $\shift_N$ sends $\SSKT(\alpha,\varphi)\hookrightarrow \SSKT(0^N\times \alpha,\varphi)$, 
where $0^N\times \alpha$ is the weak composition by adding $N$ 0's at the beginning of $\alpha$.
Now, we are ready to prove the $\varphi$-flagged analog of \cite[Thm.~5.10]{AS}.

\begin{proposition}\label{cor 4.6}
    Given $r^{(\bullet)}\in\rfc(w,\varphi)$ and any $i>0$, if $e_i(r^{(\bullet)})\neq 0$, then  $\hat{P}\left(e_i(r^{(\bullet)})\right)=\hat{P}(r^{(\bullet)})$ and $\hat Q(e_i(r^{(\bullet)}))=e_i\big(\hat Q(r^{(\bullet)})\big)$.
\end{proposition}

\begin{proof}

    By \cite[Thm.~4.11]{MS}, $e_i$ preserves the Coxeter--Knuth equivalence relation. Thus, we have 
    $\hat{P}\left(e_i(r^{(\bullet)})\right)=\hat{P}(r^{(\bullet)})$.
    Now let $N$ be a large positive integer such that the images of $\shift_N$ lie in the relevant subsets corresponding to the standard flag. Let $\hat Q(0)=0$, $\shift_N(0)=0$ and consider the following diagram:
    \[
        \begin{tikzcd}[row sep=1em, column sep=.24em]
            \rfc(w,\varphi) \arrow[dd,hook,"\shift_N"] \arrow[rr,"\hat Q"] \arrow[rd,"e_i"]&  & \bigsqcup_\alpha \SSKT(\alpha,\varphi) \arrow[rd,"e_i"]\arrow[ dd,hook,"\shift_N"near start]\\
            & \rfc(w,\varphi)\sqcup\{0\} \arrow[rr,"\hat Q\quad",crossing over] && \bigsqcup_\alpha \SSKT(\alpha,\varphi) \sqcup\{0\} \arrow[dd,hook,"\shift_N"] \\
            \rfc(1_N \times w) \arrow[rr,"\hat Q" near start] \arrow[rd,"e_i"] && \bigsqcup_{\alpha}\SSKT(0^N\times \alpha) \arrow[rd,"e_i"]\\
            & \rfc(1_N \times w) \sqcup\{0\}\arrow[rr,"\hat Q"] \arrow[from=uu, crossing over, "\shift_N" near start,hook]  && \bigsqcup_{\alpha}\SSKT(0^N\times \alpha) \sqcup\{0\}
        \end{tikzcd}.
    \]
    The proposition is equivalent to saying that the top face of this diagram commutes.
    The bottom face commutes by \cite[Thm.~5.10]{AS}.
    Crystal operators $e_i$ for both reduced factorizations and key tableaux only depend on the relative order of the entries, so $\shift_N$ commutes with crystal operators $e_i$. 
    Also, $\shift_N$ commutes with $\hat Q(-)$ 
because $\hat{P}(\shift_N(r^{(\bullet)}))$ is obtained from $\hat{P}(r^{(\bullet)})$ by shifting all boxes up by $N$ rows and adding $N$ to all entries.
Hence, all vertical faces commute.
    Since the rightmost vertical arrow is injective, the top face must commute.
\end{proof}

The following theorem is a generalization of \cite[Thm.~3.14]{AS}, which asserts that $\SSKT(\alpha)$ is a Demazure crystal with character $\kappa_\alpha$.
    For positive integers $t>s$ define $\pi_{t\downarrow s}:=\pi_{t-1}\pi_{t-2}\cdots\pi_{s}$
and  let $\pi_{t\downarrow t}:=\id$. 

\begin{theorem}\label{demazure crystal sskt}
    Suppose $\alpha$ is a weak composition with length $k$ and $\varphi$ is any flag. 
    Then
    $\SSKT(\alpha,\varphi)$ is a Demazure crystal (of type $\gl_n$ for any $n \geq \varphi(k)$) with character
    $\kappa_{(\alpha,\varphi)}=\pi_{\varphi(1)\downarrow 1}\pi_{\varphi(2)\downarrow 2}\cdots \pi_{\varphi(k)\downarrow k}(\kappa_{\alpha})$.
\end{theorem}
    
\begin{proof}
   
    By \cite[Thm.~3.14]{AS}, $\SSKT(\alpha)$ is a Demazure crystal. Suppose  $\varphi(k)\leq n$  and $\SSKT(\alpha)$ can be embedded into $\rf_n(w)$ for some $w\in S_\infty$; hence, $\YR_\alpha(w)\neq \varnothing$.
   Choose any $T \in \YR_\alpha(w)$. Define $\cC$ to be the set of $\rho^{(\bullet)} \in \rfc(w)$
   with $\hat{P}(\rho^{(\bullet)}) = T$. Let $\cC'$ be the set of  $\rho^{(\bullet)} \in \rfc(w,\varphi)$
   with $\hat{P}(\rho^{(\bullet)}) = T$.
   By Proposition~\ref{flagged bijection} and Proposition~\ref{cor 4.6}, the map 
   $\hat Q(-)$
   is a crystal isomorphism $\SSKT(\alpha)\cong \cC$ and $\SSKT(\alpha,\varphi)\cong \cC'$.
   Therefore $\cC$ is a Demazure crystal embedded in $\rf_n(w)$, so $\cC=\fkD_{\sigma}^{\rf_n(w)}\{u^{(\bullet)}\}$ for some $\sigma \in S_n$ and some highest weight element $u^{(\bullet)}$. 
   
 Let $\fkD_i = \fkD_i^{\rf_n(w)}$. 
    Since the $e_i$ operators preserve the weak EG insertion tableau,
    Corollary~\ref{corollary 4.5}
    implies that $\cC'=\fkD_{t_1\downarrow 1}\fkD_{t_2\downarrow 2}\cdots \fkD_{t_n\downarrow n}\cC$
   where $t_i=\min\{n,\varphi(i)\}$. 
      Hence, $\cC'$ is also a Demazure crystal, and
         the Demazure character formula \cite[Thm.~13.7]{BS} implies that
   \be\label{ch-eq}\ch(\SSKT(\alpha,\varphi)) = \ch(\cC')=\pi_{t_1\downarrow 1} \cdots \pi_{t_n\downarrow n}\ch(\cC)
   =\pi_{t_1\downarrow 1}\cdots \pi_{t_n\downarrow n}(\kappa_{\alpha}).\ee
   Since $\alpha$ has length $k$, $\pi_{t_{k+1}\downarrow k+1}\cdots\pi_{t_n\downarrow n}(\kappa_\alpha)=\kappa_\alpha$ by properties in \S\ref{2.1}. Also, $t_i=\varphi(i)$ for all $1\leq i\leq k$ since $\varphi(i)\leq \varphi(k)\leq n$. Therefore, \eqref{ch-eq} reduces to $\ch(\SSKT(\alpha,\varphi))=\pi_{\varphi(1)\downarrow 1}\cdots \pi_{\varphi(k)\downarrow n}(\kappa_{\alpha})$.
\end{proof}

As an application of the theorem, we derive a recurrence for $\kappa_{(\alpha,\varphi)}$.

\begin{theorem}\label{recursion}
Let $\varphi$ be a flag.
    If $\varphi$ is strictly increasing, then $\kappa_{(\alpha,\varphi)}=\kappa_\beta$ where $\beta_j=
        \alpha_{i}$ if $\varphi(i)=j$ for some $i$ or $0$ otherwise. If $\varphi$ is not strictly increasing and $i$ is the smallest index with $\varphi(i)=\varphi(i+1)$, then 
        \begin{equation}\label{equation 2}
            \kappa_{(\alpha,\varphi)}=\begin{cases}
                \kappa_{(\alpha,\varphi-\mathbf{e}_i)}          & \text{ if }\alpha_i\leq \alpha_{i+1}, \\
                \kappa_{(\alpha\cdot s_i,\varphi-\mathbf{e}_i)} & \text{ if }\alpha_i>\alpha_{i+1}.
            \end{cases}
        \end{equation}
\end{theorem}

\begin{proof}
    By Theorem~\ref{demazure crystal sskt}, we have $\kappa_{(\alpha,\varphi)}=\pi_{\varphi(1) \downarrow 1}\cdots \pi_{\varphi(n)\downarrow n}\kappa_{\alpha}$.
    If $\varphi$ is strictly increasing, then $\kappa_{(\alpha,\varphi)}$ expands into $\kappa_\beta$ directly using the recursive property in \S\ref{2.1}.

    Assume $\varphi$ is not strictly increasing and 
    $i$ is the smallest integer such that $\varphi(i)=\varphi(i+1)=N$.
 Notice that 
    \[\begin{aligned} \pi_{N\downarrow i}\pi_{N\downarrow(i+1)} &= (\pi_{N-1} \cdots \pi_{i+1} \pi_i) (\pi _{N-1}\cdots \pi_{i+2}\pi_{i+1}) \\&=
     (\pi_{N-2} \cdots \pi_{i+1} \pi_i)  (\pi _{N-1}\cdots \pi_{i+2}\pi_{i+1})\pi_i=
    \pi_{(N-1)\downarrow i}\pi_{N\downarrow(i+1)} \pi _i\end{aligned}\]
    since both expressions give reduced words for the same permutation when every ``$\pi$'' is replaced by ``$s$''.
Substituting this identity and noting that 
 all subscripts in $\pi_{\varphi(i+2)\downarrow(i+2)} \cdots \pi_{\varphi(n)\downarrow n}$ are at least $i+2$,
   we deduce that
    \[\kappa_{(\alpha,\varphi)}=\pi_{\varphi(1) \downarrow 1}\cdots \pi_{(\varphi(i)-1)\downarrow i}\pi_{\varphi(i+1)\downarrow(i+1)}\cdots \pi_{\varphi(n)\downarrow n}(\pi_i\kappa_{\alpha})\]
    and this becomes the desired identity by the recursive property in \S\ref{2.1}.
\end{proof}
 
 Continue to assume that    $n\geq \varphi(\ell(\alpha))$, so that $\SSKT(\alpha,\varphi)$ is a $\gl_n$-crystal.
 
\begin{corollary}\label{thm 4.15}
    Suppose 
    $\varphi$ is a non-standard flag with $i\leq n$ the smallest positive integer such that $\varphi(i)=\varphi(i+1)$. 
Then the crystal $\SSKT(\alpha,\varphi)$ is isomorphic to $\SSKT(\alpha\cdot s_i,\varphi-\mathbf{e}_i)$ if $\alpha_i>\alpha_{i+1}$ or to $\SSKT(\alpha,\varphi-\mathbf{e}_i)$ if $\alpha_i\leq \alpha_{i+1}$.
\end{corollary}

\begin{proof}
    By Theorems~\ref{demazure crystal sskt} and  \ref{recursion}, the relevant crystals are Demazure crystals with the same character. Hence, they are isomorphic.
\end{proof}

Reiner and Shimozono defined \defn{$\varphi$-flagged key polynomials} associated to $\alpha$ as $\sum_{u^{\bullet}\in\sW(\alpha,\varphi)}x^{\weight(u^{\bullet})}$ for a certain set $\sW(\alpha,\varphi)$; see \cite{RS}. As a final application, we can now prove Theorem~\ref{thm 1.4} from the introduction.


\begin{proof}[Proof of Theorem~\ref{thm 1.4}]

If $\varphi$ is the standard flag, $\kappa_\alpha=\sum_{u^{\bullet}\in\sW(\alpha,\varphi)}x^{\weight(u^{\bullet})}$ \cite{LS2}. 
    In the proof of {\cite[Thm.~21]{RS}},
    Reiner and Shimozono observed the following list of recursive properties of $\sW(\alpha,\varphi)$.
    When $\varphi$ is strictly increasing,
    we have $\sW(\alpha,\varphi)=\sW(\beta)$,
    where $\beta_{j}=\alpha_i$ if $j=\varphi(i)$ and $\beta_j=0$ otherwise.
    When $\varphi$ is not strictly increasing,
    assume $i\in\ZZ_{>0}$ is minimal such that $\varphi(i)=\varphi(i+1)$.
    If $\alpha_{i}\leq \alpha_{i+1}$ then $\sW(\alpha,\varphi)=\sW(\alpha,\varphi-\mathbf{e}_i)$.
    If $\alpha_{i}>\alpha_{i+1}$ 
    then there is a bijection between $\sW(\alpha,\varphi)$ and $\sW(\alpha\cdot s_i,\varphi)=\sW(\alpha\cdot s_i,\varphi-\mathbf{e}_i)$.
    Hence, $\sW(\alpha,\varphi)$ satisfies the same recursive relations involving the flag $\varphi$ as $\SSKT{(\alpha,\varphi)}$ in Corollary~\ref{thm 4.15}.
\end{proof}



\section*{Acknowledgements}
I am grateful to my advisor, Eric Marberg, for invaluable advice and discussions. I also thank Sami Assaf for several helpful conversations.

\printbibliography

@article{AS,
  author  = {Sami Assaf and Anne Schilling},
  journal = {Algebraic Combinatorics},
  number  = {2},
  pages   = {225--247},
  title   = {A Demazure crystal construction for Schubert polynomials},
  volume  = {1},
  year    = {2018},
}

@article{Assaf18,
  author        = {Sami Assaf},
  journal       = {Transactions of the American Mathematical Society},
  number        = {12},
  pages         = {8777--8796},
  title         = {Nonsymmetric Macdonald polynomials and a refinement of Kostka-Foulkes polynomials},
  volume        = {370},
  year          = {2018},
}

@article{Assaf19,
  author  = {Sami Assaf},
  journal = {Algebraic Combinatorics},
  number  = {2},
  pages   = {359--385},
  title   = {A generalization of Edelman--Greene insertion for Schubert polynomials},
  volume  = {4},
  year    = {2021},
}

@article{BJS,
  author  = {Sara Billey and William Jockusch and Richard P. Stanley},
  journal = {Journal of Algebraic Combinatorics},
  number  = {4},
  pages   = {345--374},
  title   = {Some combinatorial properties of Schubert polynomials},
  volume  = {2},
  year    = {1993},
}

@book{BS,
  author    = {Daniel Bump and Anne Schilling},
  publisher = {Word Scientific},
  title     = {Crystal bases: representations of combinatorics},
  year      = {2017},
}

@article{Dem,
  author  = {Michel Demazure},
  journal = {Annales Scientifiques de l'{\'E}cole Normale Sup{\'e}rieure},
  pages   = {53--88},
  series  = {4},
  title   = {D\'esingularisation des vari\'et\'es de Schubert g\'en\'eralis\'ees},
  volume  = {7},
  year    = {1974},
}

@book{EC2,
  author    = {Richard P. Stanley},
  editor    = {},
  publisher = {Cambridge University Press},
  title     = {Enumerative Combinatorics},
  volume    = {2},
  year      = {1999},
}

@article{EG,
  author  = {Paul Edelman and Curtis Greene},
  journal = {Advances in Mathematics},
  number  = {1},
  pages   = {42--99},
  title   = {Balanced tableaux},
  volume  = {63},
  year    = {1987},
}

@article{FS,
  author  = {Sergey Fomin and Richard P. Stanley},
  journal = {Advances in Mathematics},
  pages   = {196--207},
  title   = {Schubert polynomials and the nilCoxeter algebra},
  volume  = {193},
  year    = {1994},
}

@article{K,
  author  = {Masaki Kashiwara},
  journal = {Duke Mathematical Journal},
  number  = {3},
  pages   = {839--858},
  title   = {The crystal base and Littelmann's refined Demazure character formula},
  volume  = {71},
  year    = {1993},
}

@unpublished{Ku,
  author     = {Siddheswar Kundu},
  eprint     = {2309.10709},
  eprinttype = {arxiv},
  title      = {Demazure crystal structure for flagged reverse plane partitions},
  year       = {2023},
}

@article{LS,
  author  = {Alain Lascoux and Marcel-Paul Sch\"utzenberger},
  journal = {Comptes Rendus de l'Acad{\'e}mie des Sciences Paris},
  number  = {13},
  pages   = {447--450},
  title   = {Polyn\^omes de Schubert},
  volume  = {294},
  year    = {1982},
}

@book{LS2,
  author    = {Alain Lascoux and Marcel-Paul Sch\"utzenberger},
  chapter   = {Keys \& standard bases},
  pages     = {125--144},
  publisher = {Springer},
  series    = {IMA Volumes in Mathematics and its Applications},
  title     = {Invariant theory and tableaux},
  volume    = {19},
  year      = {1990},
}

@book{M,
  author    = {Laurent Manivel},
  publisher = {American Mathematical Society},
  title     = {Symmetric Functions, Schubert Polynomials and Degeneracy Loci},
  year      = {2001},
}

@article{MS,
  author  = {Jennifer Morse and Anne Schilling},
  journal = {International Mathematics Research Notices},
  number  = {8},
  pages   = {2239--2294},
  title   = {Crystal approach to affine Schubert calculus},
  year    = {2016},
}

@article{RS,
  author  = {Victor Reiner and Mark Shimozono},
  journal = {Journal of Combinatorial Theory, Series A},
  number  = {1},
  pages   = {107--143},
  title   = {Key polynomials and a flagged Littlewood--Richardson rule},
  volume  = {70},
  year    = {1995},
}

@article{S,
  author  = {Richard P. Stanley},
  journal = {European Journal of Combinatorics},
  number  = {5},
  pages   = {359--372},
  title   = {On the number of reduced decompositions of elements of Coxeter groups},
  volume  = {4},
  year    = {1984},
}

@misc{W,
  author     = {Jiayi Wen},
  eprint     = {2309.17025},
  eprinttype = {arxiv},
  title      = {Demazure crystals for flagged key polynomials},
  year       = {2023},
}

\end{document}